\documentclass[reqno]{amsart}

\usepackage{amssymb,amsthm,bm,mathrsfs,enumerate,array}
\usepackage{hyperref}

\newcommand{\arxiv}[1]{\href{http://arxiv.org/abs/#1}{arXiv:#1}}

\theoremstyle{plain}
\newtheorem{thm}{Theorem}[section]
\newtheorem{lem}[thm]{Lemma}
\newtheorem{supp}[thm]{Supplement}
\theoremstyle{definition}
\newtheorem{ex}[thm]{Example}
\theoremstyle{remark}
\newtheorem{rem}[thm]{Remark}

\newcommand{\M}{M}
\newcommand{\dM}{M^*}
\newcommand{\T}{T}
\newcommand{\Mh}{\bm{M}_H}
\newcommand{\dMh}{\bm{M}_H^*}
\newcommand{\Th}{\bm{T}_H}
\newcommand{\Me}{\bm{M}}
\newcommand{\dMe}{\bm{M}^*}
\newcommand{\Te}{\bm{T}}

\begin{document}

\title[An Assmus--Mattson theorem]{An Assmus--Mattson theorem for codes over commutative association schemes}
\author{John Vincent S. Morales}
\address{Graduate School of Information Sciences, Tohoku University, Sendai, Japan}
\email{moralesjohnvince@ims.is.tohoku.ac.jp}
\author{Hajime Tanaka}
\address{\href{http://www.math.is.tohoku.ac.jp/}{Research Center for Pure and Applied Mathematics}, Graduate School of Information Sciences, Tohoku University, Sendai, Japan}
\email{htanaka@tohoku.ac.jp}
\begin{abstract}
We prove an Assmus--Mattson-type theorem for block codes where the alphabet is the vertex set of a commutative association scheme (say, with $s$ classes).
This in particular generalizes the Assmus--Mattson-type theorems for $\mathbb{Z}_4$-linear codes due to Tanabe (2003) and Shin, Kumar, and Helleseth (2004), as well as the original theorem by Assmus and Mattson (1969).
The weights of a code are $s$-tuples of non-negative integers in this case, and the conditions in our theorem for obtaining $t$-designs from the code involve concepts from polynomial interpolation in $s$ variables.
The Terwilliger algebra is the main tool to establish our results.
\end{abstract}

\maketitle

\section{Introduction}

We begin by recalling the famous \emph{Assmus--Mattson theorem} which relates linear codes and combinatorial designs:

\begin{thm}[Assmus and Mattson {\cite[Theorem 4.2]{AM1969JCT}}]\label{original AMT}
Let $C$ be a linear code of length $n$ over $\mathbb{F}_q$ with minimum weight $\delta$.
Let $C^{\perp}$ denote the dual code of $C$, with minimum weight $\delta^*$.
Suppose that an integer $t$ $(1\leqslant t\leqslant n)$ is such that there are at most $\delta-t$ weights of $C^{\perp}$ in $\{1,2,\dots,n-t\}$, or such that there are at most $\delta^*-t$  weights of $C$ in $\{1,2,\dots,n-t\}$.
Then the supports of the words of any fixed weight in $C$ form a $t$-design (with possibly repeated blocks).
\end{thm}

\noindent
We remark that \cite[Theorem 4.2]{AM1969JCT} also includes a criterion for obtaining simple $t$-designs, but we will not pay much attention in this paper to the simplicity of the resulting designs.
There are several proofs and strengthenings of Theorem \ref{original AMT}; see, e.g., \cite{CDS1991IEEE,CD1993SIAM,Simonis1995LAA,Bachoc1999DCC,Tanabe2001DCC,KP2003DCC,Tanaka2009EJC}.
The purpose of this paper is to establish a theorem which unifies many of the known generalizations and extensions of Theorem \ref{original AMT}.

Constructing $t$-designs from codes received renewed interest when Gulliver and Harada \cite{GH1999DM} and Harada \cite{Harada1998JCD} found new $5$-designs by computer from the lifted Golay code of length $24$ over $\mathbb{Z}_4$ (among others).
Their constructions were later explained and generalized further by Bonnecaze, Rains, and Sol\'{e} \cite{BRS2000JSPI}.
Motivated by these results, Tanabe \cite{Tanabe2000IEEE} obtained an Assmus--Mattson-type theorem for $\mathbb{Z}_4$-linear codes with respect to the symmetrized weight enumerator.
Tanabe's theorem can indeed capture the $5$-designs from the lifted Golay code over $\mathbb{Z}_4$, but the conditions in his theorem involve finding the ranks of matrices having quite complicated entries, so that it is hard to verify the conditions without the help of a computer.
Tanabe \cite{Tanabe2003DCC} then presented a simpler version of his theorem, and we can easily check its conditions by hand for the lifted Golay code over $\mathbb{Z}_4$.

To be somewhat concrete, by an \emph{Assmus--Mattson-type} theorem, we mean in this paper a theorem which enables us to find $t$-designs by just looking at some kind of weight enumerator of a code (plus a bit of extra information in some cases, e.g., linearity).
Such a theorem is not always the best way to estimate the parameter $t$ of the resulting designs as it does not take into account the structure of the code at all (cf.~Remark \ref{Goethals codes}), but instead it has a great advantage in its wide range of applicability.

When working with the Hamming weight enumerator as in Theorem \ref{original AMT}, we are dealing with codes in the \emph{Hamming association schemes}. 
(Formal definitions will begin in Section \ref{sec: preliminaries}.)
Hamming association schemes are examples of \emph{metric} and \emph{cometric} association schemes, and Theorem \ref{original AMT} can be interpreted and generalized from this point of view; cf.~\cite{Tanaka2009EJC}.
On the other hand, in situations where we focus on a more complicated type of weight enumerator of a block code as in \cite{Tanabe2000IEEE,Tanabe2003DCC}, we think of the code in question (say, of length $n$) as lying in a structure much finer than a Hamming association scheme; that is to say, the alphabet itself naturally becomes the vertex set of a commutative association scheme with $s$ classes where $s\geqslant 2$, and we consider its \emph{extension} of length $n$.
Hamming association schemes are the same thing as extensions of $1$-class (i.e., trivial) association schemes, but if $s\geqslant 2$ then its extensions are no longer metric nor cometric.

In this paper, we prove a general Assmus--Mattson-type theorem for codes in extensions of arbitrary commutative association schemes.
Our main results are Theorem \ref{general AMT} and Supplements \ref{1st supp}--\ref{upper bound}.
In general, the weights of a code take the form $\alpha=(\alpha_1,\alpha_2,\dots,\alpha_s)$, where the $\alpha_i$ are non-negative integers such that $\sum_{i=1}^s\alpha_i\leqslant n$.
We count the \emph{number} of weights in a given interval when $s=1$ as in Theorem \ref{original AMT}, but if $s\geqslant 2$ then instead we speak of the \emph{minimal degree} of subspaces of the polynomial ring $\mathbb{R}[\xi_1,\xi_2,\dots,\xi_s]$ which allow unique Lagrange interpolation with respect to those weights (which are lattice points in $\mathbb{R}^s$) contained in a given region.
When specialized to the case of $\mathbb{Z}_4$-linear codes with the symmetrized weight enumerator as in \cite{Tanabe2000IEEE,Tanabe2003DCC}, the association scheme on the alphabet $\mathbb{Z}_4$ has $2$ classes $R_1$ and $R_2$, together with the identity class $R_0$, defined by
\begin{equation*}
	(x,y)\in R_i \, \Longleftrightarrow \, y-x=\pm i \ (\bmod 4) \quad (x,y\in\mathbb{Z}_4)
\end{equation*}
for $i\in\{0,1,2\}$, and our results give a slight extension of Tanabe's theorem in \cite{Tanabe2003DCC}.
The Assmus--Mattson-type theorem for $\mathbb{Z}_4$-linear codes with the Hamming weight enumerator due to Shin, Kumar, and Helleseth \cite{SKH2004DCC} can also be recovered.
To prove our results, we make heavy use of the representation theory of the \emph{Terwilliger algebra} \cite{Terwilliger1992JAC,Terwilliger1993JACa,Terwilliger1993JACb}, which is a non-commutative semisimple matrix $\mathbb{C}$-algebra attached to each vertex of an association scheme.
See, e.g., \cite{Schrijver2005IEEE,GST2006JCTA,Tanaka2009EJC,BBST2015EJC} for more applications of the Terwilliger algebra to coding theory and design theory.

The layout of this paper is as follows.
Section \ref{sec: preliminaries} collects necessary notation, definitions, and results concerning commutative association schemes.
In Section \ref{sec: main results}, we state our main results.
Section \ref{sec: proofs} is devoted to their proofs.
Finally, we discuss a number of examples in Section \ref{sec: examples}.

\section{Preliminaries}\label{sec: preliminaries}

We refer the reader to \cite{Delsarte1973PRRS,BI1984B,BCN1989B,DL1998IEEE,MT2009EJC} for more background information.
In this paper, $\mathbb{N}$ will denote the set of \emph{non-negative} integers:
\begin{equation*}
	\mathbb{N}=\{0,1,2,\dots\}.
\end{equation*}

\subsection{Commutative association schemes and their Terwilliger algebras}

Let $X$ be a finite set, and let $V$ be a complex vector space with a distinguished basis $\{\hat{x}:x\in X\}$ and a Hermitian inner product $\langle\hat{x},\hat{y}\rangle=\delta_{xy}$ $(x,y\in X)$.
For every subset $C$ of $X$, we let $\hat{C}=\sum_{x\in C}\hat{x}\in V$ denote its characteristic vector.
We will naturally identify $\operatorname{End}(V)$ with the $\mathbb{C}$-algebra of complex matrices with rows and columns indexed by $X$.
The adjoint (or conjugate-transpose) of $A\in\operatorname{End}(V)$ will be denoted by $A^{\dagger}$.
Let $\mathcal{R}=\{R_0,R_1,\dots,R_s\}$ be a set of non-empty binary relations on $X$.
For each $i$, let $A_i\in\operatorname{End}(V)$ be the $0$-$1$ adjacency matrix of the graph $(X,R_i)$ (directed, in general).
The pair $(X,\mathcal{R})$ is called a \emph{commutative association scheme} with $s$ classes if
\begin{enumerate}[({A}S1)]
\item\label{AS1} $A_0=I$, the identity matrix;
\item\label{AS2} $\sum_{i=0}^sA_i=J$, the all ones matrix;
\item\label{AS3} $A_i^{\dagger}\in\{A_0,A_1,\dots,A_s\}$ for $0\leqslant i\leqslant s$;
\item\label{AS4} $A_iA_j=A_jA_i\in \M :=\sum_{k=0}^s \mathbb{C}A_k$ for $0\leqslant i,j\leqslant s$.
\end{enumerate}

For the rest of this paper, we will always assume that $(X,\mathcal{R})$ is a commutative association scheme with $s$ classes.
It follows from (AS\ref{AS1}), (AS\ref{AS2}), and (AS\ref{AS4}) that the linear subspace $\M$ of $\operatorname{End}(V)$ is an $(s+1)$-dimensional commutative $\mathbb{C}$-algebra, called the \emph{Bose--Mesner algebra} of $(X,\mathcal{R})$.
By (AS\ref{AS3}), $\M$ is closed under $^{\dagger}$, so that it is semisimple and has a basis $\{E_i\}_{i=0}^s$ consisting of the primitive idempotents, i.e., $E_iE_j=\delta_{ij}E_i$, $\sum_{i=0}^sE_i=I$.
We will always set
\begin{equation*}
	E_0=|X|^{-1}J.
\end{equation*}
We note that the $E_i$ are Hermitian positive semidefinite matrices.
By (AS\ref{AS2}), $\M$ is also closed under entrywise (or \emph{Hadamard} or \emph{Schur}) multiplication, denoted $\circ$.
The $A_i$ are the primitive idempotents of $\M$ with respect to this multiplication, i.e., $A_i\circ A_j=\delta_{ij}A_i$, $\sum_{i=0}^sA_i=J$.

The \emph{intersection numbers} $p_{ij}^k$ and the \emph{Krein parameters} $q_{ij}^k$ $(0 \leqslant i,j,k \leqslant s)$ of $(X,\mathcal{R})$ are defined by the equations
\begin{equation*}
	A_i A_j = \sum_{k=0}^s p_{ij}^k A_k, \qquad E_i \circ E_j = |X|^{-1} \sum_{k=0}^s q_{ij}^k E_k.
\end{equation*}
Clearly, the $p_{ij}^k$ are non-negative integers.
On the other hand, since $E_i\circ E_j$ (being a principal submatrix of $E_i\otimes E_j$) is positive semidefinite, it follows that the $q_{ij}^k$ are real and non-negative.

The change-of-basis matrices $P$ and $Q$ are defined by
\begin{equation}\label{eigenmatrices}
	A_i=\sum_{j=0}^sP_{ji}E_j, \qquad E_i=|X|^{-1}\sum_{j=0}^sQ_{ji}A_j.
\end{equation}
In particular,
\begin{equation}\label{P, Q are invertible}
	PQ=QP=|X|I.
\end{equation}
We refer to $P$ and $Q$ as the \emph{first} and the \emph{second eigenmatrix} of $(X,\mathcal{R})$, respectively.
Note that $P_{0i}$ is the degree (both in and out) of the regular graph $(X,R_i)$, and that $Q_{0i}$ is the rank of $E_i$.
Moreover, we have
\begin{equation}\label{0th columns}
	P_{i0}=Q_{i0}=1 \quad (0\leqslant i\leqslant s).
\end{equation}

We recall the Terwilliger algebra.
Fix a ``base vertex'' $x_0\in X$, and define the diagonal matrices $E_i^*=E_i^*(x_0)$, $A_i^*=A_i^*(x_0)$ $(0\leqslant i\leqslant s)$ in $\operatorname{End}(V)$ by
\begin{equation*}
	(E_i^*)_{xx}=(A_i)_{x_0x}, \quad (A_i^*)_{xx}=|X|(E_i)_{x_0x} \quad (x\in X).
\end{equation*}
Note that $E_i^*E_j^*=\delta_{ij}E_i^*$, $\sum_{i=0}^sE_i^*=I$, and moreover
\begin{equation*}
	A_i^*A_j^*=\sum_{k=0}^s q_{ij}^k A_k^*, \qquad A_i^*=\sum_{j=0}^s Q_{ji} E_j^*.
\end{equation*}
The $E_i^*$ and the $A_i^*$ are called the \emph{dual idempotents} and the \emph{dual adjacency matrices} of $(X,\mathcal{R})$ with respect to $x_0$, respectively.
They form two bases of the \emph{dual Bose--Mesner algebra} $\dM=\dM(x_0)$ of $(X,\mathcal{R})$ with respect to $x_0$.
The \emph{Terwilliger} (or \emph{subconstituent}) \emph{algebra} $\T=\T(x_0)$ of $(X,\mathcal{R})$ with respect to $x_0$ is the $\mathbb{C}$-subalgebra of $\operatorname{End}(V)$ generated by $\M$ and $\dM$ \cite{Terwilliger1992JAC,Terwilliger1993JACa,Terwilliger1993JACb}.
The following are relations in $\T$ (cf.~\cite[Lemma 3.2]{Terwilliger1992JAC}):
\begin{equation}\label{triple product relations}
	E_i^*A_jE_k^*=0 \, \Longleftrightarrow \, p_{ij}^k=0; \qquad E_iA_j^*E_k=0 \, \Longleftrightarrow \, q_{ij}^k=0.
\end{equation}
Since $\T$ is closed under $^{\dagger}$, it is semisimple and any two non-isomorphic irreducible $\T$-modules in the \emph{standard module} $V$ are orthogonal.
Define a partition
\begin{equation}\label{partition}
	X=X_0 \sqcup X_1 \sqcup \dots \sqcup X_s
\end{equation}
by
\begin{equation*}
	X_i=\{x\in X:(x_0,x)\in R_i\} \quad (0\leqslant i\leqslant s).
\end{equation*}
Then, since $\hat{X}_i=A_i^{\dagger}\hat{x}_0=E_i^*\hat{X}$ for every $0\leqslant i\leqslant s$, it is immediate to see that the $(s+1)$-dimensional subspace
\begin{equation*}
	\sum_{i=0}^s \mathbb{C}\hat{X}_i=\M\hat{x}_0=\dM\hat{X}
\end{equation*}
is an irreducible $\T$-module, called the \emph{primary} $\T$-\emph{module}.
It is the unique irreducible $\T$-module in $V$ containing the $1$-dimensional subspaces $E_0V$ and $E_0^*V$.

Let $C$ be a subset of $X$.
To avoid triviality, we call $C$ a \emph{code} if $1<|C|<|X|$.
For the moment, assume that $C$ is a code.
The \emph{inner distribution} of $C$ is the vector $a=(a_0,a_1,\dots,a_s)\in\mathbb{R}^{s+1}$ defined by
\begin{equation*}
	a_i=|C|^{-1}\langle \hat{C},A_i\hat{C}\rangle=|C|^{-1}\cdot |R_i\cap (C\times C)| \quad (0\leqslant i\leqslant s).
\end{equation*}
Observe that (cf.~\eqref{P, Q are invertible}, \eqref{0th columns})
\begin{equation*}
	a_0=1, \quad \sum_{i=0}^sa_i=|C|, \quad (aQ)_0=|C|, \quad \sum_{i=0}^s(aQ)_i=|X|.
\end{equation*}
Clearly, the $a_i$ are non-negative.
On the other hand, from \eqref{eigenmatrices} it follows that
\begin{equation}\label{meaning of aQ}
	\langle \hat{C},E_i\hat{C}\rangle=|X|^{-1}|C|\, (aQ)_i \quad (0\leqslant i\leqslant s).
\end{equation}
Since the $E_i$ are positive semidefinite, it follows that the $(aQ)_i$ are also non-negative.
Delsarte's famous \emph{linear programming bound} \cite{Delsarte1973PRRS} on the sizes of codes is based on this simple observation.
The vector $aQ\in\mathbb{R}^{s+1}$ is often referred to as the \emph{MacWilliams transform} of $a$.
We remark the following:
\begin{equation*}
	(aQ)_i=0 \, \Longleftrightarrow \, E_i\hat{C}=0.
\end{equation*}

\subsection{Translation association schemes}

Suppose that $X$ is endowed with the structure of an abelian group (written additively) with identity element $0$.
We call $(X,\mathcal{R})$ a \emph{translation association scheme} \cite[\S 2.10]{BCN1989B} if for all $0\leqslant i\leqslant s$ and $z\in X$, $(x,y)\in R_i$ implies $(x+z,y+z)\in R_i$.

For the rest of this section, assume that $(X,\mathcal{R})$ is a translation association scheme.
In this context, we will \emph{always} choose $0$ as the base vertex.
(Note that the automorphism group of $(X,\mathcal{R})$ is transitive on $X$.)
Observe that
\begin{equation*}
	R_i=\{(x,y)\in X\times X:y-x\in X_i\} \quad (0\leqslant i\leqslant s).
\end{equation*}
Let $X^*$ be the character group of $X$ with identity element $\iota$.
To each $\varepsilon\in X^*$ we associate the vector
\begin{equation*}
	\hat{\varepsilon}=|X|^{-1/2}\sum_{x\in X}\overline{\varepsilon(x)}\,\hat{x} \in V,
\end{equation*}
so that
\begin{equation}\label{inner product for two bases}
	\langle\hat{x},\hat{\varepsilon}\rangle=|X|^{-1/2}\varepsilon(x) \quad (x\in X,\, \varepsilon\in X^*).
\end{equation}
Note that the $\hat{\varepsilon}$ form an orthonormal basis of $V$ by the orthogonality relations for the characters.
Moreover, it follows that
\begin{equation*}
	A_i\hat{\varepsilon}=\left(\sum_{x\in X_i}\overline{\varepsilon(x)}\right)\hat{\varepsilon} \quad (0\leqslant i\leqslant s,\ \varepsilon\in X^*).
\end{equation*}
This shows that each of the $\hat{\varepsilon}$ is an eigenvector for $\M$, and hence belongs to one of the $E_iV$.
Thus, we have a partition
\begin{equation*}
	X^*=X_0^* \sqcup X_1^* \sqcup \dots \sqcup X_s^*,
\end{equation*}
given by
\begin{equation*}
	X_i^*=\{\varepsilon\in X^*:\hat{\varepsilon}\in E_iV\} \quad (0\leqslant i\leqslant s).
\end{equation*}
Note that $X_0^*=\{\iota\}$, and that
\begin{equation}\label{orthogonal projections}
	E_i=\sum_{\varepsilon\in X_i^*}\hat{\varepsilon}\,\hat{\varepsilon}^{\dagger} \quad (0\leqslant i\leqslant s).
\end{equation}
Define the set $\mathcal{R}^*=\{R_0^*,R_1^*,\dots,R_s^*\}$ of non-empty binary relations on $X^*$ by
\begin{equation*}
	R_i^*=\{(\varepsilon,\eta) \in X^* \times X^* :\eta\varepsilon^{-1}\in X_i^*\} \quad (0\leqslant i\leqslant s).
\end{equation*}
Then it follows from the orthogonality relations and \eqref{orthogonal projections} that
\begin{equation*}
	A_i^*=\sum_{(\varepsilon,\eta)\in R_i^*}\!\!\hat{\varepsilon}\,\hat{\eta}^{\dagger} \quad (0\leqslant i\leqslant s).
\end{equation*}
In other words, the matrix representing $A_i^*$ with respect to the orthonormal basis $\{\hat{\varepsilon}:\varepsilon\in X^*\}$ of $V$ is precisely the $0$-$1$ adjacency matrix of the graph $(X^*,R_i^*)$.
It turns out that the pair $(X^*,\mathcal{R}^*)$ is again a translation association scheme, called the \emph{dual} of $(X,\mathcal{R})$.
In particular, the $q_{ij}^k$ are the intersection numbers of $(X^*,\mathcal{R}^*)$, so that these are again non-negative integers in this case.
We also note that $(X^*,\mathcal{R}^*)$ has eigenmatrices $P^*=Q$ and $Q^*=P$, and that
\begin{equation*}
	P_{ji}=\sum_{x\in X_i}\overline{\varepsilon(x)} \quad (\varepsilon\in X_j^*),  \qquad Q_{ji}=\sum_{\varepsilon\in X_i^*}\varepsilon(x)  \quad (x\in X_j).
\end{equation*}
We will view $V$ together with the basis $\{\hat{\varepsilon}:\varepsilon\in X^*\}$ as the standard module for $(X^*,\mathcal{R}^*)$, and choose $\iota$ as the base vertex.

A code $C$ in $X$ is called an \emph{additive code} if it is a subgroup of $X$.
Assume for the moment that $C$ is an additive code, and let $a=(a_0,a_1,\dots,a_s)$ be its inner distribution.
Observe that
\begin{equation*}
	a_i=|X_i\cap C| \quad (0\leqslant i\leqslant s),
\end{equation*}
and hence $a$ is also called the \emph{weight distribution} of $C$ in this case.
The \emph{dual code} of $C$ is the subgroup $C^{\perp}$ in $X^*$ defined by
\begin{equation*}
	C^{\perp}=\{\varepsilon\in X^*: \varepsilon(x)=1 \ \text{for all} \ x\in C\}.
\end{equation*}
From \eqref{inner product for two bases} it follows that
\begin{equation}\label{how characteristic vectors are related}
	\hat{C}=|X|^{-1/2}\,|C|\sum_{\varepsilon\in C^{\perp}}\!\hat{\varepsilon}.
\end{equation}
In other words, $\hat{C}$ is a scalar multiple of the characteristic vector of $C^{\perp}$ with respect to the basis $\{\hat{\varepsilon}:\varepsilon\in X^*\}$.
We now observe that
\begin{equation}\label{dual distribution}
	\langle \hat{C},E_i\hat{C}\rangle=|X|^{-1}|C|^2\cdot |X_i^*\cap C^{\perp}| \quad (0\leqslant i\leqslant s).
\end{equation}
In particular, combining this with \eqref{meaning of aQ}, we have
\begin{equation*}
	|X_i^*\cap C^{\perp}|=|C|^{-1} (aQ)_i \quad (0\leqslant i\leqslant s),
\end{equation*}
so that $|C|^{-1} (aQ)$ gives the weight distribution of $C^{\perp}$.

The group operation on $X^*$ is multiplicative.
In many cases (cf.~Section \ref{sec: examples}), we fix a (non-canonical) isomorphism $X\rightarrow X^*$ $(x \mapsto \varepsilon_x)$ such that
\begin{equation}\label{symmetric form}
	\varepsilon_x(y)=\varepsilon_y(x) \quad (x,y\in X).
\end{equation}
Then the dual code of an additive code in $X$ becomes again an \emph{additive} code in $X$.

See \cite[Chapter 6]{Delsarte1973PRRS}, \cite[\S 2.10]{BCN1989B}, and \cite[\S 6]{MT2009EJC} for more details about translation association schemes.

\subsection{Extensions of commutative association schemes and Hamming association schemes}

For the rest of this paper, we will fix an integer $n$ at least $2$.
Delsarte \cite[\S 2.5]{Delsarte1973PRRS} gave a construction of a new commutative association scheme from $(X,\mathcal{R})$ with vertex set $X^n$ as follows.
For a sequence $\alpha=(\alpha_1,\alpha_2,\dots,\alpha_s)\in \mathbb{N}^s$, let $|\alpha|=\sum_{i=1}^s \alpha_i$.
For any two vertices $\bm{x}=(x_1,x_2,\dots,x_n),\bm{y}=(y_1,y_2,\dots,y_n)\in X^n$, define the \emph{composition} of $\bm{x},\bm{y}$ to be the vector $c(\bm{x},\bm{y})=(c_1,c_2,\dots,c_s)\in \mathbb{N}^s$, where
\begin{equation*}
	c_i=|\{\ell:(x_{\ell},y_{\ell})\in R_i \}| \quad (1\leqslant i\leqslant s).
\end{equation*}
It is clear that $|c(\bm{x},\bm{y})|\leqslant n$.
For every $\alpha\in\mathbb{N}^s$ with $|\alpha|\leqslant n$, define the binary relation $\bm{R}_{\alpha}$ on $X^n$ by
\begin{equation*}
	\bm{R}_{\alpha}=\{ (\bm{x},\bm{y})\in X^n\times X^n: c(\bm{x},\bm{y})=\alpha \}.
\end{equation*}
Let
\begin{equation*}
	\operatorname{Sym}^n(\mathcal{R})=\{\bm{R}_{\alpha}:\alpha\in\mathbb{N}^s,\, |\alpha|\leqslant n \}.
\end{equation*}
Then it follows that the pair $(X^n,\operatorname{Sym}^n(\mathcal{R}))$ is a commutative association scheme, called the \emph{extension} of $(X,\mathcal{R})$ of length $n$.
We will identify its standard module with $V^{\otimes n}$, so that $\hat{\bm{x}}:=\hat{x}_1\otimes\hat{x}_2\otimes\dots\otimes\hat{x}_n$ for $\bm{x}=(x_1,x_2,\dots,x_n)\in X^n$.
For every $\alpha=(\alpha_1,\alpha_2,\dots,\alpha_s) \in \mathbb{N}^s$ with $|\alpha|\leqslant n$, the $0$-$1$ adjacency matrix $\bm{A}_{\alpha}\in\operatorname{End}(V^{\otimes n})$ of the graph $(X^n,\bm{R}_{\alpha})$ is then given by
\begin{equation}\label{tensor expression for adjacency matrices}
	\bm{A}_{\alpha}=\!\sum_{i_1,i_2,\dots,i_n} \!\!\! A_{i_1}\otimes A_{i_2} \otimes \dots \otimes A_{i_n},
\end{equation}
where the sum is over $i_1,i_2,\dots,i_n\in\mathbb{N}$ such that
\begin{equation*}
	\{i_1,i_2,\dots,i_n\}=\{0^{n-|\alpha|},1^{\alpha_1},2^{\alpha_2},\dots,s^{\alpha_s}\}
\end{equation*}
as multisets.
In particular, the Bose--Mesner algebra $\Me$ of $(X^n,\operatorname{Sym}^n(\mathcal{R}))$ coincides with the $n^{\mathrm{th}}$ symmetric tensor space of $M$.
Similar expressions hold for the primitive idempotents, dual idempotents, and the dual adjacency matrices of $(X^n,\operatorname{Sym}^n(\mathcal{R}))$, denoted henceforth by the $\bm{E}_{\alpha},\bm{E}_{\alpha}^*$, and the $\bm{A}_{\alpha}^*$, respectively.
For simplicity, we will always choose $\bm{x}_0:=(x_0,x_0,\dots,x_0)\in X^n$ as the base vertex.
We denote the corresponding dual Bose--Mesner algebra and the Terwilliger algebra by $\dMe$ and $\Te$, respectively.
We also consider the partition
\begin{equation*}
	X^n=\!\bigsqcup_{\substack{\alpha\in\mathbb{N}^s \\ |\alpha|\leqslant n}} \! (X^n)_{\alpha}
\end{equation*}
corresponding to \eqref{partition}, i.e.,
\begin{equation*}
	(X^n)_{\alpha}=\{\bm{x}\in X^n:(\bm{x}_0,\bm{x})\in \bm{R}_{\alpha}\}.
\end{equation*}

Let $\{\bm{e}_i:1\leqslant i\leqslant s\}$ be the standard basis of $\mathbb{R}^s$.
Then in view of \eqref{0th columns}, we have
\begin{equation}\label{standard generators}
	\bm{A}_{\bm{e}_i}=\sum_{\substack{\alpha\in\mathbb{N}^s \\ |\alpha|\leqslant n}} \!\left(\sum_{j=0}^s \alpha_jP_{ji}\right) \! \bm{E}_{\alpha}, \quad \bm{A}_{\bm{e}_i}^*=\sum_{\substack{\alpha\in\mathbb{N}^s \\ |\alpha|\leqslant n}} \!\left(\sum_{j=0}^s \alpha_jQ_{ji}\right) \! \bm{E}_{\alpha}^*,
\end{equation}
where $\alpha_0:=n-|\alpha|$.
More generally, Mizukawa and Tanaka \cite{MT2004PAMS} described the eigenmatrices of $(X^n,\operatorname{Sym}^n(\mathcal{R}))$ in terms of certain $s$-variable hypergeometric orthogonal polynomials which generalize the Krawtchouk polynomials.
See also \cite{IT2012TAMS,Iliev2012CM}.
Let $\bm{p}_{\alpha\beta}^{\gamma}$ (resp.~$\bm{q}_{\alpha\beta}^{\gamma}$) denote the intersection numbers (resp.~Krein parameters) of $(X^n,\operatorname{Sym}^n(\mathcal{R}))$.
Then, for all $1\leqslant i\leqslant s$ and $\beta,\gamma\in\mathbb{N}^s$ with $|\beta|,|\gamma|\leqslant n$, we have
\begin{equation}\label{recurrence}
	\bm{p}_{\bm{e}_i\beta}^{\gamma}\!\ne 0 \, \Longleftrightarrow \, \gamma\in \bigl\{\beta-\bm{e}_j+\bm{e}_k: p_{ij}^k\ne 0 \bigr\},
\end{equation}
where we set $\bm{e}_0:=0$.
A similar result holds for the $\bm{q}_{\bm{e}_i\beta}^{\gamma}$.

Let $\xi=(\xi_0,\xi_1,\dots,\xi_s)$ be a sequence of $s+1$ variables.
For every $\alpha\in \mathbb{N}^s$ with $|\alpha|\leqslant n$, we let
\begin{equation}\label{monomials}
	\xi^{\alpha}=\xi_0^{n-|\alpha|}\xi_1^{\alpha_1}\xi_2^{\alpha_2}\dots\xi_s^{\alpha_s}.
\end{equation}
Then it follows from \eqref{tensor expression for adjacency matrices} that
\begin{equation*}
	\left(\sum_{i=0}^s \xi_i A_i \right)^{\!\!\otimes n}\!\!=\sum_{\substack{\alpha\in\mathbb{N}^s \\ |\alpha|\leqslant n}} \!\xi^{\alpha} \bm{A}_{\alpha},
\end{equation*}
and similarly for the $\bm{E}_{\alpha}$.
Observe that
\begin{equation*}
	\sum_{i=0}^s \xi_i E_i=|X|^{-1}\sum_{i=0}^s (\xi Q^{\mathsf{T}})_i A_i.
\end{equation*}
Combining these comments, we have (cf.~\cite{Tarnanen1987P,Godsil2010M})
\begin{equation}\label{generating function identity}
	\sum_{\substack{\alpha\in\mathbb{N}^s \\ |\alpha|\leqslant n}} \!\xi^{\alpha} \bm{E}_{\alpha}=|X|^{-n}\!\sum_{\substack{\alpha\in\mathbb{N}^s \\ |\alpha|\leqslant n}} \!(\xi Q^{\mathsf{T}})^{\alpha} \bm{A}_{\alpha}.
\end{equation}
(Here, we extend the notation \eqref{monomials} to the sequence $\xi Q^{\mathsf{T}}$ as well.)

Now, let $C$ be a code in $X^n$ with inner distribution $a=(a_{\alpha})_{\alpha\in\mathbb{N}^s,\,|\alpha|\leqslant n}$.
Consider the polynomial $w_C(\xi)$ in $\mathbb{R}[\xi]=\mathbb{R}[\xi_0,\xi_1,\dots,\xi_s]$ defined by
\begin{equation*}
	w_C(\xi)=\sum_{\substack{\alpha\in\mathbb{N}^s \\ |\alpha|\leqslant n}} \!a_{\alpha} \xi^{\alpha}.
\end{equation*}
Note that $w_C(\xi)$ is homogeneous of degree $n$.
From \eqref{generating function identity} it follows that
\begin{equation}\label{change of variables in weight enumerator}
	|C|^{-1}\!\sum_{\substack{\alpha\in\mathbb{N}^s \\ |\alpha|\leqslant n}} \! \langle \hat{C},\bm{E}_{\alpha}\hat{C}\rangle \xi^{\alpha}=|X|^{-n}\, w_C(\xi Q^{\mathsf{T}}).
\end{equation}
Hence we can read which of the $\bm{E}_{\alpha}\hat{C}$ vanish from the expansion of $w_C(\xi Q^{\mathsf{T}})$.

Suppose for the moment that $(X,\mathcal{R})$ is a translation association scheme, and that $C$ is an additive code in $X^n$.
In this case, $w_C(\xi)$ is called the \emph{weight enumerator} of $C$.
It should be remarked that $(X^n,\operatorname{Sym}^n(\mathcal{R}))$ and $(X^{*n},\operatorname{Sym}^n(\mathcal{R}^*))$ are dual to each other.
By \eqref{dual distribution} and \eqref{change of variables in weight enumerator} we have (cf.~\cite{Godsil2010M})
\begin{equation*}
	w_{C^{\perp}}(\xi)=|C|^{-1}\,w_C(\xi Q^{\mathsf{T}}).
\end{equation*}
This generalizes the well-known \emph{MacWilliams identity}.

In proving our results, we also need to consider a special fusion of $(X^n,\operatorname{Sym}^n(\mathcal{R}))$ called the \emph{Hamming association scheme} $H(n,|X|)$, which is defined to be the extension of length $n$ of the $1$-class association scheme $(X,\{R_0,(X\times X)\backslash R_0\})$.
Observe that $H(n,|X|)$ has $n$ classes, and that the associated matrices as well as the partition of the vertex set $X^n$ are parametrized by the integers $0,1,\dots,n$, i.e., $\bm{A}_i,\bm{E}_i,\bm{E}_i^*,\bm{A}_i^*$, and also $(X^n)_i$ ($0\leqslant i\leqslant n$).
We denote the corresponding Bose--Mesner algebra, the dual Bose--Mesner algebra, and the Terwilliger algebra by $\Mh,\dMh$, and $\Th$, respectively.
Note that
\begin{equation}\label{two generators}
	\bm{A}_1=\sum_{i=0}^n \theta_i \bm{E}_i, \quad \bm{A}_1^*=\sum_{i=0}^n \theta_i^* \bm{E}_i^*,
\end{equation}
where
\begin{equation*}
	\theta_i=\theta_i^*=n(|X|-1)-|X|i \quad (0\leqslant i\leqslant n).
\end{equation*}

Below we collect important facts about the irreducible $\Th$-modules, most of which can be found in Terwilliger's lecture notes \cite{Terwilliger2010N}.
See also \cite[\S 5.1]{TTW2016pre}.
(Some of the results hold in the wider class of \emph{metric} and \emph{cometric} association schemes.)

\begin{lem}\label{properties of irreducible modules}
Let $W$ be an irreducible $\Th$-module.
\begin{enumerate}[(i)]
\item $\bm{A}_1\bm{E}_i^*W\subset \bm{E}_{i-1}^*W+\bm{E}_i^*W+\bm{E}_{i+1}^*W$ $(0\leqslant i\leqslant n)$, where $\bm{E}_{-1}^*=\bm{E}_{n+1}^*=0$.
\item $\bm{A}_1^*\bm{E}_iW\subset \bm{E}_{i-1}W+\bm{E}_iW+\bm{E}_{i+1}W$ $(0\leqslant i\leqslant n)$, where $\bm{E}_{-1}=\bm{E}_{n+1}=0$.
\item There are non-negative integers $r$ and $d$ such that
\begin{equation}\label{range of r and d}
	n-2r \leqslant d \leqslant n-r,
\end{equation}
and
\begin{equation*}
	\dim \bm{E}_i^*W=\dim\bm{E}_iW=\begin{cases} 1 & \text{if} \ r\leqslant i\leqslant r+d, \\ 0 & \text{otherwise}, \end{cases} \quad (0\leqslant i\leqslant n).
\end{equation*}
\item $\bm{E}_i^*\bm{A}_1\bm{E}_j^*W\ne 0$ if $|i-j|=1$ \ $(r\leqslant i,j\leqslant r+d)$.
\item $\bm{E}_i\bm{A}_1^*\bm{E}_jW\ne 0$ if $|i-j|=1$ \ $(r\leqslant i,j\leqslant r+d)$.
\end{enumerate}
\end{lem}

\noindent
The integers $r$ and $d$ in (iii) above are called the \emph{endpoint} and the \emph{diameter} of $W$, respectively.
The integer $2r+d-n$ is called the \emph{displacement} \cite{Terwilliger2005GC} of $W$.
From \eqref{range of r and d} it follows that
\begin{equation*}
	0 \leqslant 2r+d-n \leqslant n.
\end{equation*}
For every $0\leqslant c\leqslant n$, let $U_c$ be the span of the irreducible $\Th$-modules in $V^{\otimes n}$ with displacement $c$.
Then we have
\begin{equation*}
	V^{\otimes n}=\bigoplus_{c=0}^n U_c.
\end{equation*}
This is called the \emph{displacement decomposition} of $V^{\otimes n}$.
Terwilliger \cite{Terwilliger2010N} showed that
\begin{equation}\label{structure of U_0}
	U_0=(\mathbb{C}\hat{x}_0+\mathbb{C}\hat{X})^{\otimes n}.
\end{equation}

\section{Main results}\label{sec: main results}

We recall some concepts from polynomial interpolation; cf.~\cite{GS2000ACM}.
Let $S$ be a finite set of points in $\mathbb{R}^s$.
A linear subspace $\mathscr{L}$ of the polynomial ring $\mathbb{R}[\xi_1,\xi_2,\dots,\xi_s]$ is called an \emph{interpolation space} with respect to $S$ if, for every $f\in \mathbb{R}[\xi_1,\xi_2,\dots,\xi_s]$, there exists a unique $g\in\mathscr{L}$ such that $f(z)=g(z)$ for all $z=(z_1,z_2,\dots,z_s)\in S$.
It is called a \emph{minimal degree} interpolation space if, moreover, this $g$ always satisfies $\deg f\geqslant \deg g$.

Let $\mathscr{M}(S)$ denote a minimal degree interpolation space with respect to $S$, and let
\begin{equation*}
	\mu(S)=\max\{\deg f:f\in\mathscr{M}(S)\}.
\end{equation*}
We note that $\mathscr{M}(S)$ exists; see Theorem \ref{existence of minimal one} below.
Observe also that $\mu(S)$ is well-defined, i.e., it is independent of the choice of $\mathscr{M}(S)$.

We retain the notation of Section \ref{sec: preliminaries}.
For $\bm{x}=(x_1,x_2,\dots,x_n)\in X^n$, let
\begin{equation*}
	\operatorname{supp}(\bm{x})=\{\ell:x_{\ell}\ne x_0\}\subset\{1,2,\dots,n\}.
\end{equation*}
We call $\operatorname{supp}(\bm{x})$ the \emph{support} of $\bm{x}$ (with respect to $\bm{x}_0=(x_0,x_0,\dots,x_0)$).

\begin{thm}\label{general AMT}
Let $C$ be a code in $X^n$.
Let
\begin{equation*}
	S_r=\{\alpha\in\mathbb{N}^s:r\leqslant |\alpha|\leqslant n-r,\, \bm{E}_{\alpha}^*\hat{C}\ne 0\} \quad (1\leqslant r\leqslant\lfloor n/2\rfloor),
\end{equation*}
and let
\begin{equation*}
	\delta^*=\min\{i\ne 0:\bm{E}_i\hat{C}\ne 0\}.
\end{equation*}
Suppose that an integer $t$ $(1\leqslant t\leqslant n)$ is such that
\begin{equation}\label{inequality}
	\mu(S_r)<\delta^*-r \quad (1\leqslant r\leqslant t).
\end{equation}
Then the multiset
\begin{equation}\label{these are t-designs}
	\{\operatorname{supp}(\bm{x}):\bm{x}\in (X^n)_{\alpha}\cap C\}
\end{equation}
is a $t$-design (with block size $|\alpha|$) for every $\alpha\in\mathbb{N}^s$ with $|\alpha|\leqslant n$.
\end{thm}

We use Theorem \ref{general AMT} together with the following ``supplements''.

\begin{supp}\label{1st supp}
Let $C$ be a code in $X^n$.
Assume that we are given in advance a set $K\subset\mathbb{N}^s$ such that the multiset \eqref{these are t-designs} is a $t$-design for every $\alpha\in K$.
Then the condition \eqref{inequality} in Theorem \ref{general AMT} may be replaced by
\begin{equation*}
	\mu(S_r\backslash K)<\delta^*-r \quad (1\leqslant r\leqslant t).
\end{equation*}
\end{supp}

We call a subset $C$ of $X^n$ a \emph{weakly $t$-balanced array}\footnote{This term is meant as only provisional; cf.~\cite{SC1973P}.} over $(X,\mathcal{R})$ (with respect to $\bm{x}_0$) if, for any $\Lambda\subset\{1,2,\dots,n\}$ and $\gamma\in\mathbb{N}^s$ such that $|\gamma|\leqslant |\Lambda| \leqslant t$, the number
\begin{equation*}
	\bigl|\bigl\{\bm{x}\in C:(x_{\ell})_{\ell\in\Lambda}\in (X^{|\Lambda|})_{\gamma}\bigr\}\bigr|
\end{equation*}
depends only on $|\Lambda|$ and $\gamma$.

Recall that, when considering a translation association scheme, we always choose the identity as the base vertex.

\begin{supp}\label{2nd supp}
Suppose that $(X,\mathcal{R})$ is a translation association scheme, and that $C$ is an additive code in $X^n$.
Assume that we are given in advance a set $L\subset\mathbb{N}^s$ such that, for every $\alpha\in L$, $(X^{*n})_{\alpha}\cap C^{\perp}$ is a weakly $t$-balanced array over $(X^*,\mathcal{R}^*)$.
Then the scalar $\delta^*$ in Theorem \ref{general AMT} may be replaced by
\begin{equation}\label{modified dual distance}
	\min\{|\alpha|:0\ne \alpha\in\mathbb{N}^s\backslash L,\, \bm{E}_{\alpha}\hat{C}\ne 0\}.
\end{equation}
\end{supp}

We note that, in the particular case where $\alpha\in\mathbb{N}^s$ is of the form $\alpha=h\bm{e}_i$ for some $h>0$, the condition that $(X^{*n})_{\alpha}\cap C^{\perp}$ is a weakly $t$-balanced array over $(X^*,\mathcal{R}^*)$ is equivalent to saying that the multiset
\begin{equation*}
	\{\operatorname{supp}(\bm{\varepsilon}):\bm{\varepsilon}\in (X^{*n})_{\alpha}\cap C^{\perp}\}
\end{equation*}
is a $t$-design.

Supplement \ref{upper bound} below was inspired by \cite[Theorem 2]{Tanabe2003DCC}, and allows us to estimate $\mu(S)$, and hence $t$, by geometrical considerations; see Section \ref{sec: examples}.
It is a general result about minimal degree interpolation spaces, so that we give a proof right after the statement.

\begin{supp}\label{upper bound}
Let $S$ be a finite set of points in $\mathbb{R}^s$.
Suppose that there are real scalars $z_{i\ell}$ $(1\leqslant i\leqslant s,\, \ell\in\mathbb{N})$, a positive integer $m$, and a linear automorphism $\sigma\in\operatorname{GL}(\mathbb{R}^s)$ such that $z_{ik}\ne z_{i\ell}$ whenever $k\ne \ell$, and that
\begin{equation}\label{transformation}
	\sigma(S)\subset \left\{(z_{1\alpha_1},z_{2\alpha_2},\dots,z_{s\alpha_s})\in\mathbb{R}^s:\alpha\in\mathbb{N}^s,\,|\alpha|\leqslant m \right\}.
\end{equation}
Then $\mu(S)\leqslant m$.
\end{supp}

\begin{proof}
We abbreviate $z_{\alpha}:=(z_{1\alpha_1},z_{2\alpha_2},\dots,z_{s\alpha_s})$.
Let $\Sigma$ denote the RHS in \eqref{transformation}.
It suffices to show that $\mu(\Sigma)\leqslant m$.
To this end, we construct an interpolation space with respect to $\Sigma$ with maximum degree at most $m$ as follows.
Let $\alpha\in\mathbb{N}^s$ be given with $|\alpha|\leqslant m$, and assume that we have constructed polynomials
\begin{equation*}
	f_{\beta}\in\mathbb{R}[\xi_1,\xi_2,\dots,\xi_s] \quad (\beta\in\mathbb{N}^s,\,|\alpha|<|\beta|\leqslant m)
\end{equation*}
such that $\deg f_{\beta}\leqslant m$ and
\begin{equation*}
	f_{\beta}(z_{\gamma})=\delta_{\beta\gamma} \quad (\gamma\in\mathbb{N}^s,\,|\gamma|\leqslant m).
\end{equation*}
Define $g_{\alpha}\in\mathbb{R}[\xi_1,\xi_2,\dots,\xi_s]$ by
\begin{equation*}
	g_{\alpha}=\prod_{i=1}^s\prod_{\ell=0}^{\alpha_i-1}\frac{\xi_i-z_{i\ell}}{z_{i\alpha_i}-z_{i\ell}},
\end{equation*}
and let
\begin{equation*}
	f_{\alpha}=g_{\alpha}-\!\!\sum_{\substack{\beta\in\mathbb{N}^s \\ |\alpha|<|\beta|\leqslant m}} \!\!\! g_{\alpha}(z_{\beta})f_{\beta}.
\end{equation*}
Then $\deg f_{\alpha}\leqslant m$, and it is easy to see that
\begin{equation}\label{interpolants}
	f_{\alpha}(z_{\gamma})=\delta_{\alpha\gamma} \quad (\gamma\in\mathbb{N}^s,\,|\gamma|\leqslant m).
\end{equation}
Thus, by induction we obtain polynomials $f_{\alpha}$ with $\deg f_{\alpha}\leqslant m$ satisfying \eqref{interpolants} for all $\alpha\in\mathbb{N}^s$ with $|\alpha|\leqslant m$.
It is clear that the subspace
\begin{equation*}
	\sum_{\substack{\alpha\in\mathbb{N}^s \\ |\alpha|\leqslant m}} \mathbb{R}f_{\alpha} \subset \mathbb{R}[\xi_1,\xi_2,\dots,\xi_s]
\end{equation*}
is an interpolation space with respect to $\Sigma$, and the proof is complete.
\end{proof}

We end this section by recalling a construction of a minimal degree interpolation space $\mathscr{M}(S)$ due to de Boor and Ron \cite{BR1990CA,BR1992MZ}.
See also \cite[\S 3]{GS2000ACM}.
For every non-zero element $f=\sum_{i=0}^{\infty}f_i$ in the ring of formal power series $\mathbb{R}[[\xi_1,\xi_2,\dots,\xi_s]]$ where $f_i$ is homogeneous of degree $i$, let
\begin{equation*}
	f_{\downarrow}=f_{i_0},
\end{equation*}
where $i_0=\min\{i:f_i\ne 0\}$.
We conventionally set $0_{\downarrow}:=0$.

\begin{thm}[\cite{BR1990CA,BR1992MZ}]\label{existence of minimal one}
Let $S$ be a finite set of points in $\mathbb{R}^s$.
Let $\mathscr{E}$ be the subspace of $\mathbb{R}[[\xi_1,\xi_2,\dots,\xi_s]]$ spanned by the exponential functions
\begin{equation*}
	\exp\left(\sum_{i=1}^s z_i\xi_i \right) \quad ((z_1,z_2,\dots,z_s)\in S).
\end{equation*}
Then the subspace
\begin{equation*}
	\sum_{f\in\mathscr{E}}\mathbb{R}f_{\downarrow}\subset\mathbb{R}[\xi_1,\xi_2,\dots,\xi_s]
\end{equation*}
is a minimal degree interpolation space with respect to $S$.
\end{thm}

\noindent
Theorem \ref{existence of minimal one} immediately leads to the following formula for $\mu(S)$ which is well suited for computer calculations:

\begin{supp}
For every finite set $S$ of points in $\mathbb{R}^s$, the scalar $\mu(S)$ equals the smallest $m\in\mathbb{N}$ for which the polynomials
\begin{equation*}
	\sum_{k=0}^m \left(\sum_{i=1}^s z_i\xi_i \right)^{\!\!k} \quad ((z_1,z_2,\dots,z_s)\in S)
\end{equation*}
are linearly independent.
\end{supp}

\noindent
(Note that we just discarded the irrelevant factors $1/(k!)$ in the Taylor polynomials of these exponential functions.)

\section{Proofs}\label{sec: proofs}

We begin by proving a few preliminary lemmas.
Recall the space $U_0$ spanned by the irreducible $\Th$-modules in $V^{\otimes n}$ with displacement $0$.
We let $\pi_{U_0}:V^{\otimes n}\rightarrow U_0$ denote the orthogonal projection onto $U_0$.
Note that $\pi_{U_0}$ is a $\Th$-homomorphism.

\begin{lem}\label{orthogonal}
The primary $\Te$-module $\Me\hat{\bm{x}}_0$ is orthogonal to every non-primary irreducible $\Th$-module in $U_0$.
\end{lem}

\begin{proof}
Let $u_0=\hat{x}_0$ and $u_1=\hat{X}-\hat{x}_0$.
For every $\tau=(\tau_1,\tau_2,\dots,\tau_n)\in\{0,1\}^n$, let
\begin{equation*}
	\bm{u}_{\tau}=u_{\tau_1}\otimes u_{\tau_2}\otimes \dots\otimes u_{\tau_n} \in \bm{E}_{|\tau|}^*U_0,
\end{equation*}
where $|\tau|=\sum_{\ell=1}^n \tau_{\ell}$ denotes the weight of $\tau$.
The $\bm{u}_{\tau}$ form an orthogonal basis of $U_0$ by \eqref{structure of U_0}, and we have
\begin{equation}\label{norm of u}
	||\bm{u}_{\tau}||^2=(|X|-1)^{|\tau|}.
\end{equation}
For every $\alpha\in\mathbb{N}^s$ with $|\alpha|\leqslant n$, using \eqref{tensor expression for adjacency matrices} we routinely have
\begin{equation}\label{inner product}
	\langle \bm{A}_{\alpha}\hat{\bm{x}}_0,\bm{u}_{\tau}\rangle=\delta_{|\tau|,|\alpha|} \binom{|\alpha|}{\alpha_1,\alpha_2,\dots,\alpha_s} \prod_{i=1}^s (P_{0i})^{\alpha_i},
\end{equation}
where we recall that $P_{0i}$ is the degree of $(X,R_i)$.
Since
\begin{equation*}
	\pi_{U_0}\bm{A}_{\alpha}\hat{\bm{x}}_0=\sum_{\tau\in\{0,1\}^n} \!\! ||\bm{u}_{\tau}||^{-2} \langle \bm{A}_{\alpha}\hat{\bm{x}}_0,\bm{u}_{\tau}\rangle \,\bm{u}_{\tau},
\end{equation*}
it follows from \eqref{norm of u} and \eqref{inner product} that $\pi_{U_0}\bm{A}_{\alpha}\hat{\bm{x}}_0$ is a scalar multiple of
\begin{equation*}
	\sum_{\substack{\tau\in\{0,1\}^n \\ |\tau|=|\alpha|}}\!\!\bm{u}_{\tau}=\bm{A}_{|\alpha|}\hat{\bm{x}}_0\in \Mh\hat{\bm{x}}_0.
\end{equation*}
It follows that $\pi_{U_0}\Me\hat{\bm{x}}_0=\Mh\hat{\bm{x}}_0$, as desired.
\end{proof}

\begin{lem}\label{characterization of designs}
Let $C$ be a non-empty subset of $(X^n)_k$ for some $0\leqslant k\leqslant n$.
Then the following are equivalent:
\begin{enumerate}[(i)]
\item The multiset $\{\operatorname{supp}(\bm{x}):\bm{x}\in C\}$ is a $t$-design.
\item $\bm{E}_i\pi_{U_0}\hat{C}$ is a scalar multiple of $\bm{E}_i\hat{\bm{x}}_0$ for every $0\leqslant i\leqslant t$.
\item $\hat{C}$ is orthogonal to every non-primary irreducible $\Th$-module in $U_0$ with endpoint at most $t$.
\end{enumerate}
\end{lem}

\begin{proof}
First, we show (i)\,$\Leftrightarrow$\,(ii).
To this end, we introduce another orthogonal basis of $U_0$ as follows.
Define $v_0,v_1\in V$ by
\begin{equation}\label{v_0, v_1}
	v_0=E_0\hat{x}_0=|X|^{-1}\hat{X}, \quad v_1=(I-E_0)\hat{x}_0=\hat{x}_0-|X|^{-1}\hat{X}.
\end{equation}
Note that
\begin{equation*}
	||v_0||^2=|X|^{-1}, \quad ||v_1||^2=1-|X|^{-1}, \quad \langle v_0,v_1\rangle=0.
\end{equation*}
For every $\tau=(\tau_1,\tau_2,\dots,\tau_n)\in\{0,1\}^n$, let
\begin{equation*}
	\bm{v}_{\tau}=v_{\tau_1}\otimes v_{\tau_2}\otimes \dots\otimes v_{\tau_n} \in \bm{E}_{|\tau|}U_0,
\end{equation*}
where $|\tau|=\sum_{\ell=1}^n \tau_{\ell}$.
The $\bm{v}_{\tau}$ form an orthogonal basis of $U_0$ by \eqref{structure of U_0}, and we have
\begin{equation*}
	||\bm{v}_{\tau}||^2=|X|^{-n}\,(|X|-1)^{|\tau|}.
\end{equation*}
Moreover, observe that
\begin{equation*}
	\sum_{\substack{\tau\in \{0,1\}^n \\ |\tau|=i}} \!\!\bm{v}_{\tau}=\bm{E}_i\hat{\bm{x}}_0.
\end{equation*}
By these comments and since
\begin{equation*}
	\bm{E}_i\pi_{U_0}\hat{C}=\sum_{\substack{\tau\in \{0,1\}^n \\ |\tau|=i}} \!\! ||\bm{v}_{\tau}||^{-2} \langle \hat{C},\bm{v}_{\tau} \rangle \,\bm{v}_{\tau},
\end{equation*}
it follows that (ii) holds if and only if $\langle \hat{C},\bm{v}_{\tau} \rangle$ depends only on $|\tau|$ whenever $|\tau|\leqslant t$.

Assume that (i) holds.
Let $\tau\in\{0,1\}^n$ with $|\tau|\leqslant t$.
From \eqref{v_0, v_1} it follows that
\begin{align*}
	\langle \hat{C},\bm{v}_{\tau} \rangle=|X|^{-n}\sum_{i=0}^{|\tau|}&(-1)^{i}(|X|-1)^{|\tau|-i} \\
	& \times \bigl|\bigl\{\bm{x}\in C:|\operatorname{supp}(\tau)\cap\operatorname{supp}(\bm{x})|=i\bigr\}\bigr|,
\end{align*}
which is indeed a constant depending only on $|\tau|$, and hence (ii) holds.

Conversely, assume that (ii) holds.
Let $\tau\in\{0,1\}^n$ with $|\tau|=t$, and let
\begin{equation*}
	\bm{w}_{\tau}=w_{\tau_1}\otimes w_{\tau_2}\otimes \dots\otimes w_{\tau_n},
\end{equation*}
where $w_0=\hat{X}=|X|v_0$ and $w_1=\hat{x}_0=v_0+v_1$.
On the one hand, we have
\begin{equation}\label{complement is t-design}
	\langle \hat{C},\bm{w}_{\tau} \rangle = |\{\bm{x}\in C: \operatorname{supp}(\tau)\cap\operatorname{supp}(\bm{x})=\emptyset \}|.
\end{equation}
On the other hand, observe that
\begin{equation*}
	\bm{w}_{\tau}=|X|^{n-t}\sum_{\rho}\bm{v}_{\rho},
\end{equation*}
where the sum is over $\rho\in\{0,1\}^n$ with $\operatorname{supp}(\rho)\subset\operatorname{supp}(\tau)$.
It follows that the common value in \eqref{complement is t-design} is independent of the choice of $\tau$, and hence (i) holds.

Next, we show (ii)\,$\Leftrightarrow$\,(iii).
Observe that
\begin{equation}\label{E_iU_0}
	\bm{E}_iU_0=\mathbb{C}\bm{E}_i\hat{\bm{x}}_0 \perp \sum_{W} \bm{E}_iW \quad (0\leqslant i\leqslant t),
\end{equation}
where the sum is over the non-primary irreducible $\Th$-modules $W$ in $U_0$ with endpoint at most $i$.
If (iii) holds, then the vectors $\bm{E}_i\hat{C}$ $(0\leqslant i\leqslant t)$ are also orthogonal to every non-primary irreducible $\Th$-module in $U_0$ with endpoint at most $t$, and hence the vector $\pi_{U_0}\bm{E}_i\hat{C}\in \bm{E}_iU_0$ vanishes on the second term of the RHS in \eqref{E_iU_0} for every $0\leqslant i\leqslant t$; in other words, (ii) holds.

Conversely, let $W$ be a non-primary irreducible $\Th$-module in $U_0$ with endpoint $r\leqslant t$, and assume that $\hat{C}$ is not orthogonal to $W$.
Let $\pi_W:V^{\otimes n}\rightarrow W$ be the orthogonal projection onto $W$.
Then we have $\pi_W\hat{C}\ne 0$.
Let
\begin{equation*}
	\ell=\min\{i:\bm{E}_i\pi_W\hat{C}\ne 0\}.
\end{equation*}
By Lemma \ref{properties of irreducible modules}\,(iii), $\bm{E}_{\ell}\pi_W\hat{C}$ spans $\bm{E}_{\ell}W$.
In view of Lemma \ref{properties of irreducible modules}\,(ii),\,(v), we have
\begin{equation*}
	\bm{E}_r(\bm{A}_1^*)^{\ell-r}\pi_W\hat{C}=\bm{E}_r(\bm{A}_1^*)^{\ell-r}\bm{E}_{\ell}\pi_W\hat{C}\ne 0.
\end{equation*}
Since $\pi_W$ is a $\Th$-homomorphism and since $\hat{C}\in\bm{E}_k^*V^{\otimes n}$, it follows from \eqref{two generators} that
\begin{equation*}
	0\ne \bm{E}_r\pi_W(\bm{A}_1^*)^{\ell-r}\hat{C}=(\theta_k^*)^{\ell-r}\bm{E}_r\pi_W\hat{C},
\end{equation*}
i.e., we must have $\ell=r$.
It follows that $\bm{E}_i\pi_{U_0}\hat{C}$ does not vanish on the second term of the RHS in \eqref{E_iU_0} when $i=r$, and hence (ii) fails to hold.
We have now shown (ii)\,$\Leftrightarrow$\,(iii), and the proof is complete.
\end{proof}

\begin{lem}\label{implication of weakly balanced designs}
Let $C$ be a non-empty subset of $(X^n)_{\alpha}$ for some $\alpha\in\mathbb{N}^s$ with $|\alpha|\leqslant n$.
Suppose that $C$ is a weakly $t$-balanced array over $(X,\mathcal{R})$.
Then
\begin{equation*}
	\bm{E}_i^*\pi_{U_0}\Me\hat{C}=\mathbb{C}\bm{A}_i\hat{\bm{x}}_0 \quad (0\leqslant i\leqslant t).
\end{equation*}
\end{lem}

\begin{proof}
We fix $\beta\in\mathbb{N}^s$ with $|\beta|\leqslant n$, and consider the vector $\bm{A}_{\beta}\hat{C}\in\Me\hat{C}$.
We use the notation in the proof of Lemma \ref{orthogonal}.
Let $\tau\in\{0,1\}^n$ with $|\tau|\leqslant t$.
We will use $'$ and $''$ to denote objects associated with the extensions of $(X,\mathcal{R})$ of lengths $|\tau|$ and $n-|\tau|$, respectively; e.g., $\bm{A}_{\gamma}'$ ($\gamma\in\mathbb{N}^s$, $|\gamma|\leqslant |\tau|$), $\bm{A}_i'$ ($0\leqslant i\leqslant |\tau|$), $\bm{x}_0'\in X^{|\tau|}$ for the former.
We understand that the coordinates of $X^{|\tau|}$ and $X^{n-|\tau|}$ are indexed by $\operatorname{supp}(\tau)$ and $\{1,2,\dots,n\}\backslash\operatorname{supp}(\tau)$, respectively.
With this notation established, we have
\begin{equation*}
	\bm{A}_{\beta}=\sum_{\nu}\bm{A}_{\nu}'\otimes \bm{A}_{\beta-\nu}'',
\end{equation*}
where the sum is over $\nu\in\mathbb{N}^s$ such that $\beta-\nu\in\mathbb{N}^s$, $|\nu|\leqslant|\tau|$, and $|\beta-\nu|\leqslant n-|\tau|$.
Observe also that
\begin{equation*}
	\bm{u}_{\tau}=\bm{A}_{|\tau|}'\hat{\bm{x}}_0' \otimes \hat{\bm{x}}_0''.
\end{equation*}
Hence we have
\begin{align}\label{inner product 2}
	\langle \bm{A}_{\beta}\hat{C},\bm{u}_{\tau} \rangle =& \sum_{\nu,\rho} g_{\nu\rho} \cdot \bigl\langle \hat{C}, (\bm{A}_{\rho}')^{\dagger}\hat{\bm{x}}_0' \otimes (\bm{A}_{\beta-\nu}'')^{\dagger} \hat{\bm{x}}_0'' \bigr\rangle \notag \\
	=& \sum_{\nu,\rho} g_{\nu\rho}\cdot \bigl|\bigl\{\bm{x}\in C:(x_{\ell})_{\ell\in\operatorname{supp}(\tau)}\in (X^{|\tau|})_{\rho}\bigr\}\bigr|,
\end{align}
where the sums are over $\nu,\rho\in\mathbb{N}^s$ such that $|\nu|,|\rho|\leqslant |\tau|$, $\beta-\nu=\alpha-\rho\in\mathbb{N}^s$, and $|\beta-\nu|\leqslant n-|\tau|$, and where we write
\begin{equation*}
	\bm{A}_{\nu}'\bm{A}_{|\tau|}'=\bm{A}_{|\tau|}'\bm{A}_{\nu}'=\sum_{\substack{\rho\in\mathbb{N}^s \\ |\rho|\leqslant |\tau|}} g_{\nu\rho}\bm{A}_{\rho}'.
\end{equation*}
By the assumption, the RHS in \eqref{inner product 2} depends only on $|\tau|\leqslant t$, and hence it follows that $\bm{E}_i^*\pi_{U_0}\bm{A}_{\beta}\hat{C}$ is a scalar multiple of $\bm{A}_i\hat{\bm{x}}_0$ for every $0\leqslant i\leqslant t$ as in the proof of Lemma \ref{orthogonal}.
We have now shown that $\bm{E}_i^*\pi_{U_0}\Me\hat{C}$ is a subspace of $\mathbb{C}\bm{A}_i\hat{\bm{x}}_0$ for $0\leqslant i\leqslant t$.
That it is non-zero and hence agrees with $\mathbb{C}\bm{A}_i\hat{\bm{x}}_0$ follows from
\begin{equation*}
	\bm{E}_i^*\pi_{U_0}J^{\otimes n}\hat{C}=|C|\bm{A}_i\hat{\bm{x}}_0.
\end{equation*}
This completes the proof.
\end{proof}

\subsection{Proof of Theorem \ref{general AMT}}

Define $\bm{D}_1^*,\bm{D}_2^*,\dots,\bm{D}_s^*\in\dMe$ by
\begin{equation*}
	\bm{D}_i^*=\sum_{\substack{\alpha\in\mathbb{N}^s \\ |\alpha|\leqslant n}} \alpha_i \bm{E}_{\alpha}^* \quad (1\leqslant i\leqslant s).
\end{equation*}
Observe that the $\bm{D}_i^*$ generate $\dMe$.
By \eqref{P, Q are invertible}, \eqref{0th columns}, and \eqref{standard generators}, for $1\leqslant j\leqslant s$ we have
\begin{align}
	\sum_{i=1}^s P_{ij} \bm{A}_{\bm{e}_i}^* &= \sum_{\substack{\alpha\in\mathbb{N}^s \\ |\alpha|\leqslant n}} \!\left(\sum_{h=0}^s \alpha_h \sum_{i=1}^s Q_{hi}P_{ij}\right) \! \bm{E}_{\alpha}^* \notag \\
	&= \sum_{\substack{\alpha\in\mathbb{N}^s \\ |\alpha|\leqslant n}} \!\left(\sum_{h=0}^s \alpha_h \bigl(|X|\delta_{hj}-Q_{h0}P_{0j}\bigr) \right) \! \bm{E}_{\alpha}^* \notag \\
	&= |X|\bm{D}_j^*-nP_{0j}\, I^{\otimes n}, \label{change of variables is affine linear}
\end{align}
where we have used $\alpha_0=n-|\alpha|$.
In particular, the $\bm{A}_{\bm{e}_i}^*$ also generate $\dMe$.

Now, fix $\alpha\in\mathbb{N}^s$ with $|\alpha|\leqslant n$.
We invoke Lemma \ref{characterization of designs} to show that the multiset \eqref{these are t-designs} is a $t$-design.
Let $W$ be a non-primary irreducible $\Th$-module in $U_0$ with endpoint $r\leqslant t$.
Recall that $W$ has diameter $n-2r$.
It suffices to show that $\bm{E}_{\alpha}^*\hat{C}$ is orthogonal to $W$.
Let $\pi_W:V^{\otimes n}\rightarrow W$ be the orthogonal projection onto $W$.
First, we show that
\begin{equation}\label{claim}
	\pi_W \bm{E}_{\alpha}^*\hat{C}\,\in\!\sum_{i=\delta^*\!-\mu_r}^{n-r}\!\!\bm{E}_iW,
\end{equation}
where $\mu_r:=\mu(S_r)$.
To this end, let $f\in \mathscr{M}(S_r)$ be such that
\begin{equation*}
	f(\beta)=\delta_{\alpha\beta} \quad (\beta\in S_r).
\end{equation*}
Observe that
\begin{equation*}
	f(\bm{D}_1^*,\bm{D}_2^*,\dots,\bm{D}_s^*)-\bm{E}_{\alpha}^*\in \sum_{\beta\notin S_r} \mathbb{R} \bm{E}_{\beta}^*.
\end{equation*}
Since
\begin{equation*}
	W\subset \sum_{i=r}^{n-r}\bm{E}_i^* V^{\otimes n}
\end{equation*}
by Lemma \ref{properties of irreducible modules}\,(iii), we have
\begin{equation}\label{1st room for improvement}
	\pi_W \bm{E}_{\beta}^*\hat{C}=0 \ \ \text{unless} \ \ \beta\in S_r,
\end{equation}
from which it follows that
\begin{equation}\label{interpolation}
	\pi_W\bm{E}_{\alpha}^*\hat{C}=\pi_W f(\bm{D}_1^*,\bm{D}_2^*,\dots,\bm{D}_s^*)\hat{C}.
\end{equation}
Let $U$ be the orthogonal complement of the primary $\Te$-module $\Me \hat{\bm{x}}_0$ in $V^{\otimes n}$, and let $\pi_U:V^{\otimes n}\rightarrow U$ be the orthogonal projection onto $U$.
Note that $\pi_U\bm{E}_0=\bm{E}_0\pi_U=0$ since $\bm{E}_0V^{\otimes n}\subset \Me \hat{\bm{x}}_0$, so that
\begin{equation}\label{dual distance}
	\pi_U\hat{C}\in\sum_{i=\delta^*}^n\bm{E}_iV^{\otimes n}.
\end{equation}
Moreover, since $\pi_U$ is a $\Te$-homomorphism and since $W\subset U$ by Lemma \ref{orthogonal}, we have
\begin{equation}\label{repeated projections}
	\pi_W\bm{B}^*\hat{C}=\pi_W\pi_U\bm{B}^*\hat{C}=\pi_W\bm{B}^*\pi_U\hat{C} \quad (\bm{B}^*\!\in\dMe).
\end{equation}
By the definition of $\mu_r$ and \eqref{change of variables is affine linear}, $f(\bm{D}_1^*,\bm{D}_2^*,\dots,\bm{D}_s^*)$ is written as a polynomial in the $\bm{A}_{\bm{e}_i}^*$ with degree at most $\mu_r$.
For any $\beta,\gamma\in\mathbb{N}^s$ with $|\beta|,|\gamma|\leqslant n$, we also have
\begin{equation*}
	\bm{E}_{\beta}\bm{A}_{\bm{e}_i}^*\bm{E}_{\gamma}=0 \quad \text{if} \ \ \bigl||\beta|-|\gamma|\bigr|>1
\end{equation*}
by virtue of \eqref{triple product relations} and (the dual of) \eqref{recurrence}.
Hence it follows from \eqref{interpolation}, \eqref{dual distance}, and \eqref{repeated projections} that
\begin{equation*}
	\pi_W\bm{E}_{\alpha}^*\hat{C}\in \pi_W\!\sum_{i=\delta^*\!-\mu_r}^n\!\!\bm{E}_iV^{\otimes n}=\sum_{i=\delta^*\!-\mu_r}^{n-r}\!\!\bm{E}_iW.
\end{equation*}
This proves \eqref{claim}.

Assume now that $\bm{E}_{\alpha}^*\hat{C}$ is not orthogonal to $W$, i.e., $\pi_W\bm{E}_{\alpha}^*\hat{C}\ne 0$.
Let 
\begin{equation*}
	\ell=\min\{i:\bm{E}_i\pi_W\bm{E}_{\alpha}^*\hat{C}\ne 0\}.
\end{equation*}
By Lemma \ref{properties of irreducible modules}\,(iii), $\bm{E}_{\ell}\pi_W\bm{E}_{\alpha}^*\hat{C}$ spans $\bm{E}_{\ell}W$.
In view of Lemma \ref{properties of irreducible modules}\,(ii),\,(v), we have
\begin{equation*}
	\bm{E}_r(\bm{A}_1^*)^{\ell-r}\pi_W\bm{E}_{\alpha}^*\hat{C}=\bm{E}_r(\bm{A}_1^*)^{\ell-r}\bm{E}_{\ell}\pi_W\bm{E}_{\alpha}^*\hat{C}\ne 0.
\end{equation*}
Since $\pi_W$ is a $\Th$-homomorphism, it follows from \eqref{two generators} that
\begin{equation*}
	0\ne \bm{E}_r\pi_W(\bm{A}_1^*)^{\ell-r}\bm{E}_{\alpha}^*\hat{C}=(\theta_{|\alpha|}^*)^{\ell-r}\bm{E}_r\pi_W\bm{E}_{\alpha}^*\hat{C},
\end{equation*}
i.e., we must have $\ell=r$.
However, this contradicts \eqref{claim} since $\delta^*-\mu_r>r$.
It follows that $\pi_W\bm{E}_{\alpha}^*\hat{C}=0$, and the proof is complete.

\subsection{Proof of Supplement \ref{1st supp}}

The most important step in the proof of Theorem \ref{general AMT} was to establish \eqref{claim}, and the first key observation \eqref{interpolation} in this process was based on \eqref{1st room for improvement}.
By Lemma \ref{characterization of designs}, \eqref{1st room for improvement} can now be improved as follows:
\begin{equation*}
	\pi_W \bm{E}_{\beta}^*\hat{C}=0 \ \ \text{unless} \ \ \beta\in S_r\backslash K.
\end{equation*}
Hence it suffices to interpolate on $S_r\backslash K$, as desired.

\subsection{Proof of Supplement \ref{2nd supp}}

At the end of the proof of Theorem \ref{general AMT}, we used \eqref{claim} and the assumption $\delta^*-\mu_r>r$ to show that $\pi_W\bm{E}_{\alpha}^*\hat{C}=0$.
Observe that we arrive at the same conclusion if we can instead prove that
\begin{equation}\label{modified claim}
	\pi_W \bm{E}_{\alpha}^*\hat{C}\,\in\!\sum_{i=r+1}^{n-r}\!\!\bm{E}_iW.
\end{equation}
Let $\delta_L^*$ denote the scalar in \eqref{modified dual distance}, and recall that we are assuming that $\delta_L^*-\mu_r>r$.
Then \eqref{dual distance} becomes
\begin{equation*}
	\pi_U\Biggl(\hat{C}-\sum_{\beta\in L}\bm{E}_{\beta}\hat{C}\Biggr)\in\sum_{i=\delta_L^*}^n\bm{E}_iV^{\otimes n},
\end{equation*}
from which it follows in the same manner that
\begin{equation}\label{modified claim 1st half}
	\pi_W \bm{F}_{\alpha}^* \Biggl(\hat{C}-\sum_{\beta\in L}\bm{E}_{\beta}\hat{C}\Biggr)\in \!\sum_{i=\delta_L^*\!-\mu_r}^{n-r}\!\!\!\!\bm{E}_iW\subset \!\sum_{i=r+1}^{n-r}\!\!\bm{E}_iW,
\end{equation}
where we abbreviate
\begin{equation*}
	\bm{F}_{\alpha}^*=f(\bm{D}_1^*,\bm{D}_2^*,\dots,\bm{D}_s^*).
\end{equation*}
On the other hand, recall that the roles of $\Me$ and $\dMe$ are interchanged when we work with the basis $\{\hat{\bm{\varepsilon}}:\bm{\varepsilon}\in X^{*n}\}$ of $V^{\otimes n}$, and observe that $\bm{E}_{\beta}\hat{C}$ is a scalar multiple of the characteristic vector of $(X^{*n})_{\beta}\cap C^{\perp}$ with respect to this basis; cf.~\eqref{how characteristic vectors are related}.
Hence, for any $\beta\in L$ and $0\leqslant i\leqslant t$,
it follows from Lemma \ref{implication of weakly balanced designs} (applied to the dual) that
\begin{align*}
	\bm{E}_i\pi_{W} \bm{F}_{\alpha}^* \bm{E}_{\beta}\hat{C} &= \bm{E}_i\pi_{W}\pi_{U_0} \bm{F}_{\alpha}^* \bm{E}_{\beta}\hat{C} \\
	&= \pi_{W}\bm{E}_i \pi_{U_0} \bm{F}_{\alpha}^* \bm{E}_{\beta}\hat{C} \\
	&\in \mathbb{C} \pi_W\bm{A}_i^*\hat{\bm{\iota}} \\
	&=0,
\end{align*}
where $\bm{\iota}=(\iota,\iota,\dots,\iota)$ is the identity of $X^{*n}$, since $\bm{A}_i^*\hat{\bm{\iota}}=|X|^{n/2}\bm{E}_i\hat{\bm{0}}$ belongs to the primary $\Th$-module $\Mh\hat{\bm{0}}$.
(Recall that $\bm{x}_0=\bm{0}=(0,0,\dots,0)$ in this context.)
Hence we have
\begin{equation}\label{modified claim 2nd half}
	\pi_{W} \bm{F}_{\alpha}^* \bm{E}_{\beta}\hat{C}\in \!\sum_{i=t+1}^{n-r}\!\!\bm{E}_iW\subset
	\!\sum_{i=r+1}^{n-r}\!\!\bm{E}_iW \quad (\beta\in L).
\end{equation}
Combining \eqref{interpolation}, \eqref{modified claim 1st half}, and \eqref{modified claim 2nd half}, we obtain \eqref{modified claim}, and this completes the proof.

\section{Examples}\label{sec: examples}

In this section, we mainly discuss additive codes over various translation association schemes (so that $x_0=0$).

\subsection{Codes with Hamming weight enumerators}
Recall that the \emph{Hamming weight} of $\bm{x}=(x_1,x_2,\dots,x_n)\in X^n$ is defined by
\begin{equation*}
	\operatorname{wt}(\bm{x})=|\{\ell:x_{\ell}\ne 0\}|.
\end{equation*}
The \emph{Hamming weight enumerator} of an additive code $C$ in $X^n$ is then defined by
\begin{equation*}
	\operatorname{hwe}_C(\xi_0,\xi_1)=\sum_{\bm{x}\in C}\xi_0^{n-\operatorname{wt}(\bm{x})}\xi_1^{\operatorname{wt}(\bm{x})}.
\end{equation*}
Thus, when working with the Hamming weight enumerator, we are considering codes over the $1$-class association scheme $(X,\{R_0,(X\times X)\backslash R_0\})$ with eigenmatrices
\begin{equation*}
	P=Q=\begin{bmatrix} 1 & |X|-1 \\ 1 & -1 \end{bmatrix}\!,
\end{equation*}
whose extension of length $n$ is the Hamming association scheme $H(n,|X|)$.
In particular, we have $\Te=\Th$ in this case.
Tanaka \cite{Tanaka2009EJC} showed the following:

\begin{thm}[{\cite[Theorem 5.2, Example 5.5]{Tanaka2009EJC}}]\label{Hamming AMT}
Let $C$ be a code in $X^n$.
Let
\begin{equation*}
	\delta=\min\{i\ne 0:\bm{E}_i^*\hat{C}\ne 0\}, \quad \delta^*=\min\{i\ne 0:\bm{E}_i\hat{C}\ne 0\}.
\end{equation*}
Suppose that an integer $t$ $(1\leqslant t\leqslant n)$ is such that, for every $1\leqslant r\leqslant t$, at least one of the following holds:
\begin{align}
	|\{i:r\leqslant i\leqslant n-r,\,\bm{E}_i^*\hat{C}\ne 0\}|&\leqslant\delta^*-r, \label{Hamming inequality} \\
	|\{i:r\leqslant i\leqslant n-r,\,\bm{E}_i\hat{C}\ne 0\}|&\leqslant\delta-r. \label{Hamming inequality*}
\end{align}
Then the multiset
\begin{equation*}
	\{\operatorname{supp}(\bm{x}):\bm{x}\in (X^n)_i\cap C\}
\end{equation*}
is a $t$-design for every $0\leqslant i\leqslant n$.
\end{thm}

Observe that Theorem \ref{Hamming AMT} strengthens the original Assmus--Mattson theorem (Theorem \ref{original AMT}).
In particular, it does not require that $C$ be linear nor additive.
The condition \eqref{Hamming inequality} agrees with \eqref{inequality} when $s=1$.
Indeed, the proof of Theorem \ref{general AMT} reduces to that of Theorem \ref{Hamming AMT} for \eqref{Hamming inequality}.
The dual argument shows the result for the case \eqref{Hamming inequality*}.
(It seems that the condition dual to \eqref{inequality} does not necessarily lead to the same conclusion as Theorem \ref{general AMT} when $s>1$.)
On the other hand, Supplements \ref{1st supp} and \ref{2nd supp} refine \cite[Remark 7.1]{Tanaka2009EJC}, and prove useful as we will see below.

\begin{ex}
The Assmus--Mattson-type theorem for additive codes over $\mathbb{F}_4$ given by Kim and Pless \cite[Theorem 2.7]{KP2003DCC} follows from Theorem \ref{Hamming AMT}, except their comment on the simplicity of the designs obtained from minimum weight codewords.
The additive group of $\mathbb{F}_4$ is isomorphic to the Klein four-group $\mathbb{Z}_2\times\mathbb{Z}_2$, and additive codes over $\mathbb{F}_4$ are the same thing as \emph{linear Kleinian codes} studied by H\"{o}hn \cite{Hohn2003MA}.
It should be noted that giving an (appropriate) inner product on $\mathbb{F}_4^n\cong(\mathbb{Z}_2\times\mathbb{Z}_2)^n$, on which concepts like self-orthogonality and self-duality depend, amounts to choosing a group isomorphism $\mathbb{Z}_2\times\mathbb{Z}_2\rightarrow(\mathbb{Z}_2\times\mathbb{Z}_2)^*$ satisfying the symmetry \eqref{symmetric form}.
This last remark applies to all examples that follow.
\end{ex}

\begin{ex}
Let $C$ be an extremal binary Type II code of length $n\equiv 8\ell\ (\mathrm{mod}\ 24)$ where $\ell\in\{0,1,2\}$.
From Theorem \ref{original AMT} (or Theorem \ref{Hamming AMT}) it follows that the words of any fixed weight in $C$ support a $t$-design with $t=5-2\ell$.
Using Bachoc's results on \emph{harmonic weight enumerators} \cite{Bachoc1999DCC}, Bannai, Koike, Shinohara, and Tagami \cite[Theorem 6, Remark 5]{BKST2006MMJ} showed that if one of these (non-trivial) designs is a $(t+1)$-design then so are the others.
This observation is also immediate from Supplement \ref{1st supp}.
We note that similar observations hold for extremal Type III codes over $\mathbb{F}_3$ and extremal Type IV codes over $\mathbb{F}_4$.
See also \cite{MN2016DCC}. 
\end{ex}

\begin{ex}
Additive codes over $\mathbb{Z}_4$ are also referred to as $\mathbb{Z}_4$-\emph{linear codes}.
For a $\mathbb{Z}_4$-linear code $C$ in $\mathbb{Z}_4^n$, let
\begin{equation*}
	C_2=(2\mathbb{Z}_4^n)\cap C,
\end{equation*}
which may also be viewed as a binary linear code (called the \emph{torsion code} of $C$) since $2\mathbb{Z}_4\cong\mathbb{Z}_2$.
We note that $\operatorname{hwe}_{C_2}$ is derived immediately from either the \emph{complete} or the \emph{symmetrized weight enumerators} of $C$; cf.~Subsection \ref{subsec: Z_4-codes}.
Shin, Kumar, and Helleseth \cite[Theorem 10]{SKH2004DCC} proved an Assmus--Mattson-type theorem for $\mathbb{Z}_4$-linear codes, and we now claim that Theorem \ref{Hamming AMT}, together with Supplements \ref{1st supp} and \ref{2nd supp}, always gives at least as good estimate on $t$ as their theorem.
First, they assume that $C_2$ and $(C^{\perp})_2$ both satisfy the conclusion of Theorem \ref{Hamming AMT}.
If a (Hamming) weight of $C$ is not a weight of $C\backslash C_2$, then the corresponding words of $C$ must all belong to $C_2$, and hence by Supplement \ref{1st supp} we can exclude that weight from the weights of $C$.
The same comment applies to $C^{\perp}$.
Second, they assume that the number of non-zero weights of the shortened code of $C^{\perp}\backslash (C^{\perp})_2$ at some $t$ coordinates is bounded above by $\delta-t$.
However, the conclusion of their theorem shows in the end that this number is equal to that of non-zero weights at most $n-t$ in $C^{\perp}\backslash (C^{\perp})_2$.
Hence it follows that this second condition is not weaker than \eqref{Hamming inequality*}.
\end{ex}

\begin{rem}\label{Goethals codes}
From the Assmus--Mattson-type theorem by Shin et al.~mentioned above (or Theorem \ref{Hamming AMT}) it follows that the words of any fixed weight in the Goethals code or its dual (a Delsarte--Goethals code) over $\mathbb{Z}_4$ of length $2^m$ with $m$ odd, support a $2$-design.
However, Shin et al.~\cite[Corollaries 7 and 8]{SKH2004DCC} showed that it is in fact a $3$-design, based on what they call an \emph{Assmus--Mattson-type approach}.
See also \cite{LRV2007FFA}.
\end{rem}

\subsection{Codes with complete/symmetrized weight enumerators}\label{subsec: Z_4-codes}

Let $C$ be an additive code over the ring $\mathbb{Z}_k$.
Besides $\operatorname{hwe}_C$, it is also important to consider the \emph{complete} and the \emph{symmetrized weight enumerators} defined respectively by
\begin{align*}
	\operatorname{cwe}_C(\xi_0,\xi_1,\dots,\xi_{k-1})&=\sum_{\bm{x}\in C}\xi_0^{n_0(\bm{x})}\xi_1^{n_1(\bm{x})}\!\!\!\dots\xi_{k-1}^{n_{k-1}(\bm{x})}, \\
	\operatorname{swe}_C(\xi_0,\xi_1,\dots,\xi_e)&=\sum_{\bm{x}\in C}\xi_0^{n_0(\bm{x})}\xi_1^{n_{\pm 1}(\bm{x})}\!\!\!\dots \xi_e^{n_{\pm e}(\bm{x})},
\end{align*}
where $e=\lfloor k/2\rfloor$,
\begin{align*}
	n_i(\bm{x}) &= |\{\ell:x_{\ell}=i\}| \quad (0\leqslant i\leqslant k-1), \\
	n_{\pm i}(\bm{x}) &= n_i(\bm{x})+n_{k-i}(\bm{x}) \quad (1\leqslant i\leqslant \lfloor (k-1)/2\rfloor),
\end{align*}
and we understand that $n_{\pm e}(\bm{x})=n_e(\bm{x})$ if $k$ is even.
Thus, for $\operatorname{cwe}_C$, the initial association scheme $(X,\mathcal{R})$ is the \emph{group association scheme} of $\mathbb{Z}_k$, which is the translation association scheme on $\mathbb{Z}_k$ defined by the partition (cf.~\eqref{partition})
\begin{equation*}
	\mathbb{Z}_k=\{0\}\sqcup\{1\}\sqcup\dots\sqcup\{k-1\},
\end{equation*}
and has eigenmatrices
\begin{equation*}
	P=\bigl[\zeta_k^{ij}\bigr]_{i,j=0}^{k-1}, \quad Q=\bigl[\zeta_k^{-ij}\bigr]_{i,j=0}^{k-1},
\end{equation*}
where $\zeta_k\in\mathbb{C}$ is a primitive $k^{\mathrm{th}}$ root of unity.
For $\operatorname{swe}_C$, $(X,\mathcal{R})$ is the association scheme of the \emph{ordinary $k$-cycle}, which is defined similarly by the partition
\begin{equation*}
	\mathbb{Z}_k=\{0\}\sqcup\{\pm 1\}\sqcup\dots\sqcup\{\pm e\},
\end{equation*}
and has eigenmatrices
\begin{equation*}
	P=Q=\bigl[(1+\delta_{0,2j})^{-1}(\zeta_k^{ij}+\zeta_k^{-ij})\bigr]_{i,j=0}^e,
\end{equation*}
where $\delta_{0,2j}$ is evaluated in $\mathbb{Z}_k$.
Extensions of the ordinary $k$-cycle are referred to as \emph{Lee association schemes} \cite{Tarnanen1987P,Sole1988P}.
We note that
\begin{align*}
	\operatorname{swe}_C(\xi_0,\xi_1,\dots,\xi_e) &= \operatorname{cwe}_C(\xi_0,\xi_1,\xi_2,\xi_3,\dots,\xi_2,\xi_1), \\
	\operatorname{hwe}_C(\xi_0,\xi_1) &= \operatorname{swe}_C(\xi_0,\xi_1,\dots,\xi_1),
\end{align*}
and that
\begin{equation*}
	\operatorname{hwe}_{C_2}(\xi_0,\xi_1)=\operatorname{swe}_C(\xi_0,0,\xi_1) \quad \text{when} \ k=4.
\end{equation*}

\begin{ex}
Our main results are in fact modeled after the Assmus--Mattson-type theorem due to Tanabe \cite[Theorem 2]{Tanabe2003DCC} for $\mathbb{Z}_4$-linear codes with respect to the symmetrized weight enumerator, so that the latter is a special case of the former.
In particular, we can easily find $5$-designs from the lifted Golay code over $\mathbb{Z}_4$ of length $24$ as discussed in \cite{Tanabe2003DCC}.
See also \cite{BRS2000JSPI}.
On the other hand, it is unclear at present whether or not Tanabe's original version of his theorem \cite[Theorem 3]{Tanabe2000IEEE} is a consequence of our results.
It would be an interesting problem to understand \cite[Theorem 3]{Tanabe2000IEEE} in terms of the irreducible $\Te$-modules; cf.~\cite{Morales2016LAA}.
\end{ex}

See \cite{HRY2001DM} for a survey on $t$-designs constructed from $\mathbb{Z}_4$-linear codes.

Below we discuss the extended quadratic residue codes $XQ_{11}$ of length $12$ over small finite fields.
That these codes support $3$-designs follows from the fact that their automorphism groups contain $\operatorname{PSL}(\mathbb{F}_{11}^2)$ and hence are $3$-homogeneous on the $12$ coordinates, but we include these examples in order to demonstrate the use of our results further.
Recall again that we only look at the weight enumerators (and linearity) of these self-dual codes.
We aim at doing the relevant computations \emph{by hand}.
The first example is a warm-up:

\begin{ex}
Consider $C=XQ_{11}$ over $\mathbb{F}_3=\mathbb{Z}_3$, which is the extended ternary Golay code.
We have $\operatorname{hwe}_C$ and $\operatorname{cwe}_C$ as follows: 

\medskip
\begin{center}
\small
\begin{tabular}{>{$}c<{$}|>{$}c<{$}>{$}c<{$}|>{$}c<{$}}
	\hline
	\operatorname{wt} & n_1 & n_2 & \text{\#words} \\
	\hline\hline
	0 & 0 & 0 & 1 \\
	\hline
	6 & 6 & 0 & 22 \\
	\cline{2-4}
	& 0 & 6 & 22 \\
	\cline{2-4}
	& 3 & 3 & 220 \\
	\hline
	9 & 6 & 3 & 220 \\
	\cline{2-4}
	& 3 & 6 & 220 \\
	\hline
	12 & 12 & 0 & 1 \\
	\cline{2-4}
	& 0 & 12 & 1 \\
	\cline{2-4}
	& 6 & 6 & 22 \\
	\hline
\end{tabular}
\end{center}
\medskip
As is well known, the words of fixed (Hamming) weight $6$ or $9$ support $5$-designs by Theorem \ref{original AMT}.
(The one with block size $9$ is the non-simple trivial design with constant multiplicity $2$.)
Set $t=3$.
We have $\delta^*=6$ and
\begin{equation*}
	S_1=S_2=S_3=\{(6,0),(0,6),(3,3),(6,3),(3,6)\}.
\end{equation*}
Observe that the words with $(n_1,n_2)=(6,3)$ and those with $(n_1,n_2)=(3,6)$ come in pairs by the correspondence $\bm{x}\mapsto -\bm{x}$, so that the words with each of these two complete weight types support the (simple!) trivial design.
Hence we may disregard them by Supplement \ref{1st supp}, i.e., we set $K=\{(6,3),(3,6)\}$.
Then $S_3\backslash K$ consists of three collinear points in $\mathbb{R}^2$, and thus we have $\mu(S_3\backslash K)=2$.
Since $2<6-3$, it follows from Theorem \ref{general AMT} that the non-trivial $5$-design with block size $6$ is partitioned into two $3$-designs (after discarding repeated blocks).
\end{ex}

\begin{ex}
Consider $C=XQ_{11}$ over $\mathbb{F}_5=\mathbb{Z}_5$.
We have $\operatorname{hwe}_C$ and $\operatorname{swe}_C$ as follows: 

\medskip
\begin{center}
\small
\begin{tabular}{>{$}c<{$}|>{$}c<{$\!\!\!}>{$}c<{$}|>{$}c<{$}}
	\hline
	\operatorname{wt} & n_{\pm 1} & n_{\pm 2} & \text{\#words} \\
	\hline\hline
	0 & 0 & 0 & 1 \\
	\hline
	6 & 3 & 3 & 440 \\
	\hline
	7 & 6 & 1 & 264 \\
	\cline{2-4}
	& 1 & 6 & 264 \\
	\hline
	8 & 4 & 4 & 2640 \\
	\hline
	9 & 7 & 2 & 1320 \\
	\cline{2-4}
	& 2 & 7 & 1320 \\
	\hline
	10 & 5 & 5 & 5544 \\
	\hline
	11 & 8 & 3 & 1320 \\
	\cline{2-4}
	& 3 & 8 & 1320 \\
	\hline
	12 & 11 & 1 & 24 \\
	\cline{2-4}
	& 1 & 11 & 24 \\
	\cline{2-4}
	& 6 & 6 & 1144 \\
	\hline
\end{tabular}
\end{center}
\medskip
We have $\delta^*=6$.
Observe that Theorem \ref{original AMT} nor Theorem \ref{Hamming AMT} cannot find designs from the supports of the codewords in this case.
On the other hand, set $t=3$, and take $\sigma\in\operatorname{GL}(\mathbb{R}^2)$ such that $\sigma(i,j)=(1/5)(2i+3j,i-j)$.
Then we have
\begin{align*}
	\sigma(S_1) &= \left\{\!\!\begin{array}{l} (6,-1), \\ (5,-1),(5,0),(5,1), \\ (4,-1),(4,0),(4,1), \\ \phantom{(3,-1),\,}(3,0),(3,1) \end{array}\!\!\right\}, \\
	\sigma(S_2) &= \left\{\!\!\begin{array}{l} (5,-1),(5,0), \\ (4,-1),(4,0),(4,1), \\ \phantom{(3,-1),\,}(3,0),(3,1) \end{array}\!\!\right\}, \\
	\sigma(S_3) &= \left\{\!\!\begin{array}{l} (5,-1), \\ (4,-1),(4,0),(4,1), \\ \phantom{(3,-1),\,}(3,0),(3,1) \end{array}\!\!\right\}.
\end{align*}
From Supplement \ref{upper bound} it follows that $\mu(S_1)\leqslant 4$ and $\mu(S_2)\leqslant 3$.
If we apply Supplement \ref{upper bound} directly to $\sigma(S_3)$ then we would only obtain $\mu(S_3)\leqslant 3$, but indeed it follows that $\mu(S_3)=2$.
To see this, let
\begin{align*}
	f_{(5,-1)} &= (\xi_1-3)(\xi_1-4)/2, \\
	f_{(4,-1)} &= -(\xi_1+\xi_2-4)(\xi_1-\xi_2-3)/2, \\
	f_{(4,1)} &= (\xi_1+\xi_2-3)(\xi_1+\xi_2-4)/2, \\
	f_{(3,0)} &= (\xi_1-4)(\xi_1+\xi_2-4), \\
	f_{(3,1)} &= -(\xi_1-4)(\xi_1+2\xi_2-3)/2, \\
	f_{(4,0)} &= 1-f_{(5,-1)}-f_{(4,-1)}-f_{(4,1)}-f_{(3,0)}-f_{(3,1)}.
\end{align*}
Then we have
\begin{equation*}
	f_{\alpha}(\beta)=\delta_{\alpha\beta} \quad (\alpha,\beta\in\sigma(S_3)),
\end{equation*}
from which it follows that the linear span of the $f_{\alpha}$ $(\alpha\in\sigma(S_3))$ is an interpolation space with respect to $\sigma(S_3)$.
This shows $\mu(S_3)=2$, as desired.
Thus, the condition \eqref{inequality} is satisfied for $r\in\{1,2,3\}$.
Theorem \ref{general AMT} now shows that the codewords of any fixed symmetrized weight type support $3$-designs.
This example tells us that looking at $\operatorname{swe}_C$ may sometimes give a better estimate on $t$ than $\operatorname{hwe}_C$, even when Supplement \ref{1st supp} is not applicable.
\end{ex}

Finally, we consider $C=XQ_{11}$ over $\mathbb{F}_4=\{0,1,\omega,\omega^2\}$.
Note that $\operatorname{cwe}_C$ makes sense by defining $n_{\omega}(\bm{x})$ and $n_{\omega^2}(\bm{x})$ in the same manner as above.
The eigenmatrices of the group association scheme of $\mathbb{F}_4$ are given by
\begin{equation*}
	P=Q=\begin{bmatrix} 1 & 1 & 1 & 1 \\ 1 & 1 & -1 & -1 \\ 1 & -1 & -1 & 1 \\ 1 & -1 & 1 & -1 \end{bmatrix}.
\end{equation*}

\begin{ex}
Consider $C=XQ_{11}$ over $\mathbb{F}_4$.
We have $\operatorname{hwe}_C$ and $\operatorname{cwe}_C$ as follows:

\medskip
\begin{center}
\small
\begin{tabular}{>{$}c<{$}|>{$}c<{$}>{$}c<{$}>{$}c<{$}|>{$}c<{$}}
	\hline
	\operatorname{wt} & n_1 & n_{\omega} & n_{\omega^2} & \text{\#words} \\
	\hline\hline
	0 & 0 & 0 & 0 & 1 \\
	\hline
	6 & 2 & 2 & 2 & 330 \\
	\hline
	7 & 5 & 1 & 1 & 132 \\
	\cline{2-5}
	& 1 & 5 & 1 & 132 \\
	\cline{2-5}
	& 1 & 1 & 5 & 132 \\
	\hline
	8 & 4 & 4 & 0 & 165 \\
	\cline{2-5}
	& 4 & 0 & 4 & 165 \\
	\cline{2-5}
	& 0 & 4 & 4 & 165 \\
	\hline
	9 & 3 & 3 & 3 & 1320 \\
	\hline
	10 & 6 & 2 & 2 & 330 \\
	\cline{2-5}
	& 2 & 6 & 2 & 330 \\
	\cline{2-5}
	& 2 & 2 & 6 & 330 \\
	\hline
	11 & 5 & 5 & 1 & 132 \\
	\cline{2-5}
	& 5 & 1 & 5 & 132 \\
	\cline{2-5}
	& 1 & 5 & 5 & 132 \\
	\hline
	12 & 12 & 0 & 0 & 1 \\
	\cline{2-5}
	& 0 & 12 & 0 & 1 \\
	\cline{2-5}
	& 0 & 0 & 12 & 1 \\
	\cline{2-5}
	& 4 & 4 & 4 & 165 \\
	\hline
\end{tabular}
\end{center}
\medskip
We have $\delta^*=6$.
Again, Theorem \ref{original AMT} cannot find designs from the supports of the codewords.
Take $\sigma\in\operatorname{GL}(\mathbb{R}^3)$ such that
\begin{equation*}
	\sigma(i,j,k)=(1/4)(2i+j+k,i+2j+k,i+j+2k).
\end{equation*}
Then we have
\begin{align*}
	\sigma(S_3) &= \left\{\!\!\begin{array}{l} (2,2,2),(2,2,3),(2,3,2),(2,3,3), \\ (3,2,2),(3,2,3),(3,3,2),(3,3,3) \end{array}\!\!\right\}, \\
	\sigma(S_2) &= \sigma(S_3)\sqcup\{(4,3,3),(3,4,3),(3,3,4)\}, \\
	\sigma(S_1) &= \sigma(S_2)\sqcup\{(4,4,3),(4,3,4),(3,4,4)\}.
\end{align*}
We claim that $\mu(S_1)\leqslant 4$ and that $\mu(S_2)=\mu(S_3)=3$.
First, it is easy to see that $\mu(S_3)=3$ as $\sigma(S_3)$ forms a cube.
Next, let
\begin{equation*}
	f_{(4,3,3)}=(\xi_1-2)(\xi_1-3)/2.
\end{equation*}
Then we have
\begin{equation*}
	f_{(4,3,3)}(\alpha)=\delta_{(4,3,3),\alpha} \quad (\alpha\in \sigma(S_2)).
\end{equation*}
We similarly define $f_{(3,4,3)}$ and $f_{(3,3,4)}$.
Recall that $\mathscr{M}(\sigma(S_3))$ denotes a minimal degree interpolation space with respect to $\sigma(S_3)$.
Then it is immediate to see that
\begin{equation*}
	\mathscr{M}(\sigma(S_3))+\mathbb{R}f_{(4,3,3)}+\mathbb{R}f_{(3,4,3)}+\mathbb{R}f_{(3,3,4)}
\end{equation*}
is an interpolation space with respect to $\sigma(S_2)$.
Since $\mu(S_3)\leqslant \mu(S_2)$, we have $\mu(S_2)=3$.
Finally, let for example
\begin{equation*}
	f_{(4,4,3)}=(\xi_1+\xi_2-4)(\xi_1+\xi_2-5)(\xi_1+\xi_2-6)(\xi_1+\xi_2-7)/24,
\end{equation*}
so that we have
\begin{equation*}
	f_{(4,4,3)}(\alpha)=\delta_{(4,4,3),\alpha} \quad (\alpha\in \sigma(S_1)),
\end{equation*}
and a similar argument establishes $\mu(S_1)\leqslant 4$, as desired.
Thus, the condition \eqref{inequality} is satisfied for $r\in\{1,2\}$ but fails for $r=3$.
Theorem \ref{general AMT} now shows that the codewords of any fixed complete weight type support $2$-designs.
Though this is not the best estimate (i.e., $t=3$), Theorem \ref{general AMT} still outperforms Theorem \ref{original AMT} for this example.
\end{ex}

\section*{Acknowledgments}

The authors thank Masaaki Harada for helpful discussions.
HT was supported in part by JSPS KAKENHI Grant No.~25400034.

\end{document}